\documentclass[12pt]{amsart}

\usepackage[english]{babel}
\usepackage[utf8]{inputenc}

\usepackage{tabu}
\usepackage{amsmath}
\usepackage{amssymb}
\usepackage{amsthm}
\usepackage{graphicx}
\usepackage{caption}
\usepackage{subcaption}
\newtheorem{theorem}{Theorem}[section]
\newtheorem{lemma}[theorem]{Lemma}

\theoremstyle{remark}

\newtheorem{question}[theorem]{Question}

\usepackage[colorinlistoftodos]{todonotes}

\title{Some experiments on Bateman-Horn}

\author{Igor Rivin}
\address{University of St Andrews School of Mathematics and Statistics}
\email{igor.rivin@st-andrews.ac.uk}
\thanks{The author would like to thank the MathOverflow community for its support, and in particular, the user \textbf{joro} who asked the question which led to this note. The author would like to thank Keith Conrad for valuable suggestions.}
\date{\today}

\begin{document}

\begin{abstract}
We describe some studies related to the frequency of prime values of integer polynomials.
\end{abstract}

\maketitle

\section{Introduction}

The Bateman-Horn conjecture  (see 
\cite{bateman1962heuristic}) states that given $m$ irreducible polynomials $f_1, \dotsc, f_m$ with integer coefficients, then the number of positive integers $k \leq x$ such that $f_i(k)$ is prime\footnote{We allow our primes to be negative.} for every $1\leq i\leq m$ is asymptotic to 
\[\frac{C}{D} \int_2^x \frac{dt}{\log^m t},\] where $D$ is the product of the degrees of the $f_i,$ while 
\[
C(f) = \prod_p \frac{1-n_p(f)/p}{1-1/p},
\]
where $n_p(f)$ is the number of solutions of $f(x) = 0 \mod p,$ where $f(x) = f_1 \cdots f_m.$

In this note, we look at the simplest case where $m=1,$ and study the behavior of the coefficient $C(f),$ where $f$ is a random monic polynomial\footnote{Such a polynomial is almost surely irreducible, though in the experiments we throw away the few exceptional cases.}  in $\mathbb{Z}[x].$ Our model of a random polynomial is one where all the coefficients (except the leading one, which is always equal to $1$) are uniformly chosen from the interval $[-N, N],$ where $N$ is typically equal to $1000.$ Further, we approximate the constant $C(f)$ by taking the products over the first several primes (first $3000$ primes in the experiment below). 
Now, there is also a finite probability that $C(f) = 0$ - this would be true if the values of $f$ are \emph{always} divisible by $p,$ for some prime $p$ - obviously the most likely such prime is $2,$ so it makes sense to look at those $f$ which don't always vanish mod $p$ for any $p$ (or, at least, any $p$ we look at).\footnote{This is the so-called Bunyakovsky condition.}

Now, what do we look at?
\begin{itemize}
\item What is the mean value of $C(f)$ over our sample space? (whether all irreducible $f$ or those $f$ satisfying the Bunyakovsky condition?
\item How are the values of $C(f)$ distributed (here, it only makes sense to look at those $f$ satisfying the Bunyakovsky condition).
\item What are the extremal values of $C,$ and what property of the polynomials involved is responsible? (we have not studied this yet).
\end{itemize}

It should be noted that in the ``orthogonal'' direction (where we have a number of \emph{linear} polynomials, which corresponds to the prime $k$-tuple conjecture), some statistical results were obtained by P.~X.~Gallagher \cite{gallagher} and E. Kowalski \cite{dummkopf}.
\section{Some experimental observations}
\label{sec:observations}

I looked at polynomials of degrees $2$ through $6.$ Here are the means:

\begin{center}
  \begin{tabular}{ l || c | r | rr }
    \hline
    degree & mean & Bunyakovsky mean & Bunyakovsky log mean \\ \hline
    2 & 1.00329 & 1.33525 & 0.0847562\\ \hline
    3 & 1.00714 & 1.38343 & 0.129013\\ \hline
    4 & 0.971798 & 1.35929 & 0.124901 \\ \hline
    5 & 1.00044 & 1.38399 &0.132433 \\ \hline
    6 & 1.02139 & 1.40398 & 0.155772 \\ \hline
    7 & 1.02448 & 1.41654 & 0.166858 \\ \hline
    8 & 1.00884 & 1.36817 & 0.120447 \\ \hline
  \end{tabular}
\end{center}
It does not seem surprising that the mean of all the irreducible polynomials is very close to $1.$ However, I am unaware of any proof of this. Here is a somewhat relevant fact:
\begin{lemma}
For each $f \in \mathbb{F}_p[x]$ let $n(f)$ be the number of zeros of $f$ counted \emph{without} multiplicity. Then, the mean value of $n(f)$ over all of $\mathbb{F}_p[x]$ is $1.$
\end{lemma}
\label{horizontalmean}
\begin{proof}
Consider the variety $V$ defined over $\mathbb{F}_p$ by 
\[
X^n + \sum_{i=0}^{n-1} Y_i X^i=0.
\]
The statement of the Lemma is equivalent to asserting that there are $p^n$ $\mathbb{F}_p$-points on $V.$ However, if we rewrite the equation of $V$ as
\[
X^n + \sum_{i=1}^{n-1} Y_i X^i= - Y_0,
\]
it becomes obvious that $V$ is nothing but the affine space $\mathbb{A}^n,$ so the result follows.
\end{proof}
Now, in our experiment we sample uniformly from polynomials in $\mathbb{Z}[x]$ with coefficients bounded above by $C$ and below by $-C,$ then reduce modulo $p.$ When $C$ is much bigger than $P$ our sample space will contain all of $\mathbb{F}_p,$ but with somewhat uneven multiplicity. When $C$ is much smaller than $p,$ our sample space contains only a small sliver of $
\mathbb{F}_p,$ so we are asserting that the $\mathbb{F}_p$ points of $V$ are very well equidistributed. If we average over primes, such a result is true.
\begin{lemma}
\label{chebot}
Let $f\in \mathbb{Z}[x]$ be irreducible over $\mathbb{Q}.$ Then, 
\[
\lim_{x\rightarrow \infty}\frac{\sum_{p\leq x}n_p(f)}{\pi(x)} = 1,
\]
where $n_p(f)$ is the number of zeros of $f$ modulo $p.$
\end{lemma}
\begin{proof}
For each prime, let $s_f = \{d_1, \dotsc, d_k\}$ be the set of degrees of irreducible factors of $f$ modulo $p.$ Chebotarev's theorem (see, e.g., \cite{stevenhagen1996chebotarev}) says that the fraction of the primes for which a  given $s_f$ occurs is the same (asymptotically) as the fraction of the elements of the Galois group of $f$ which have the cycle decomposition with lengths given by $s_f.$ If $f$ is irreducible, the Galois group of $f$ is a transitive subgroup of $S_n.$ The result now follows from Lemma 
\ref{transburn}
\end{proof}
\begin{lemma}
\label{transburn}
The average number of the fixed points of elements of a transitive permutation group equals $1.$
\end{lemma}
\begin{proof}
This follows from Burnside's lemma: for a group $G$ acting on a set $X,$ we have 
\[
|X/G| = \frac1{|G|}\sum_{g\in G}|X^g|,
\]
where $|X/G|$ is the number of orbits of the $G$ action, and $|X^g|$ is the number of fixed points of $g.$ Since $G$ acts transitively, the left hand side equals $1.$
\end{proof}
\subsection{The distribution}
As you may or may not be convinced by the charts below, the values of $C(f)$ for Bunyakovsky $f$ (that is, an $f$ which does not vanish identically modulo any prime $p$) seem to be log-normally distributed\footnote{The graphs are of the Bateman-Horn quantity $C/D,$ not of $C.$}. If true, this seems to indicate that the terms $1-n_p/p$ are independent; the distributions for different $p$ are not identical, but are the ones given by Chebotarev density. Now, independence for (very) small primes follows from the Chinese Remainder Theorem, but when the primes are large compared to the coefficients of our polynomials, that is far from clear.

\begin{figure}[!tbp]
\begin{subfigure}[b]{0.43\textwidth}
\includegraphics[width=\textwidth]{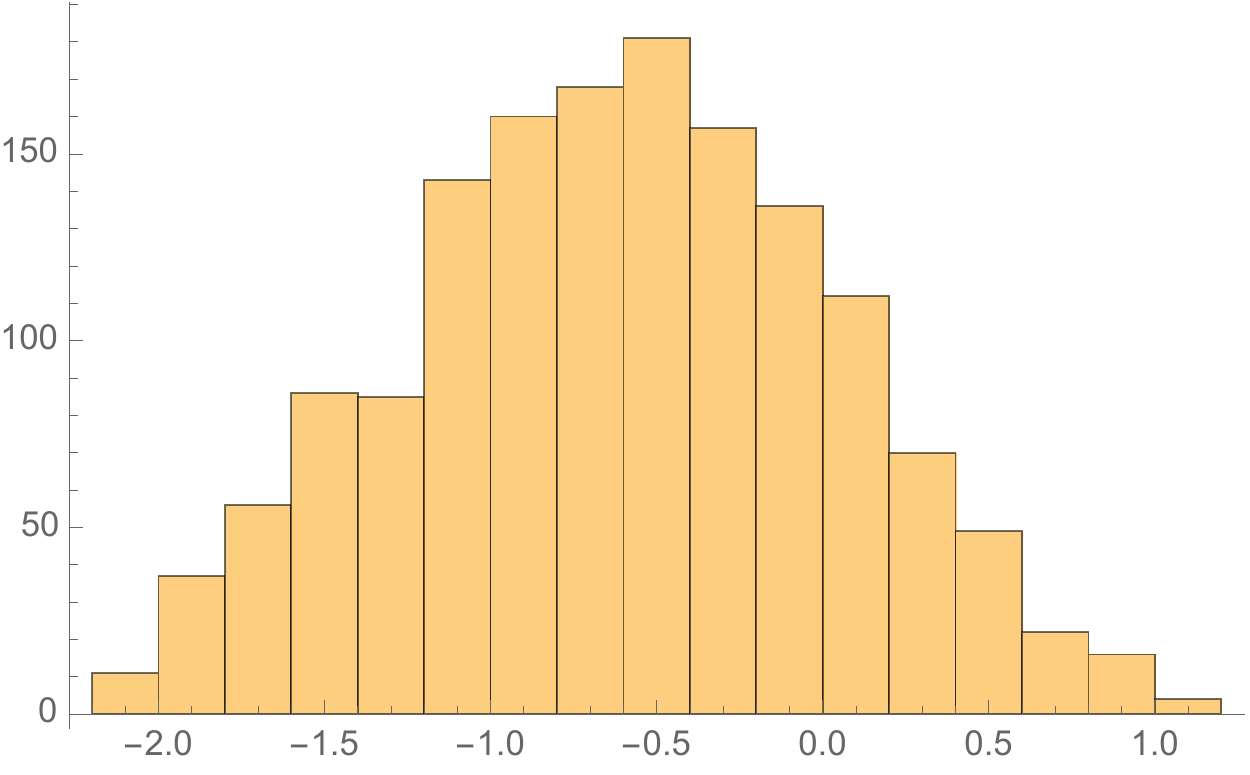}
\caption{\label{fig:histo2}Histogram of log $C(f)$}
\end{subfigure}
\hfill
\begin{subfigure}[b]{0.43\textwidth}
\includegraphics[width=\textwidth]{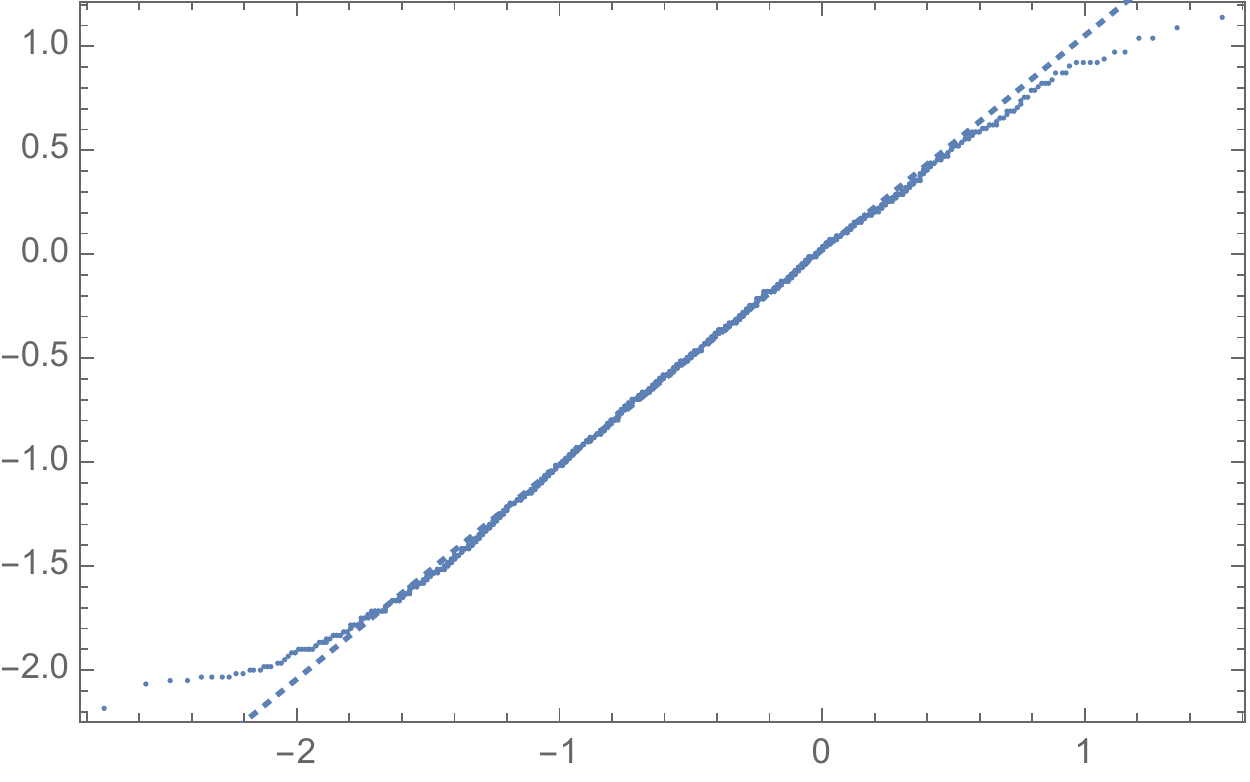}
\caption{\label{fig:qq2}quantile-quantile plot of log $C(f)$ }
\end{subfigure}
\caption{Degree 2}
\end{figure}

\begin{figure}[!tbp]
\begin{subfigure}[b]{0.43\textwidth}
\includegraphics[width=\textwidth]{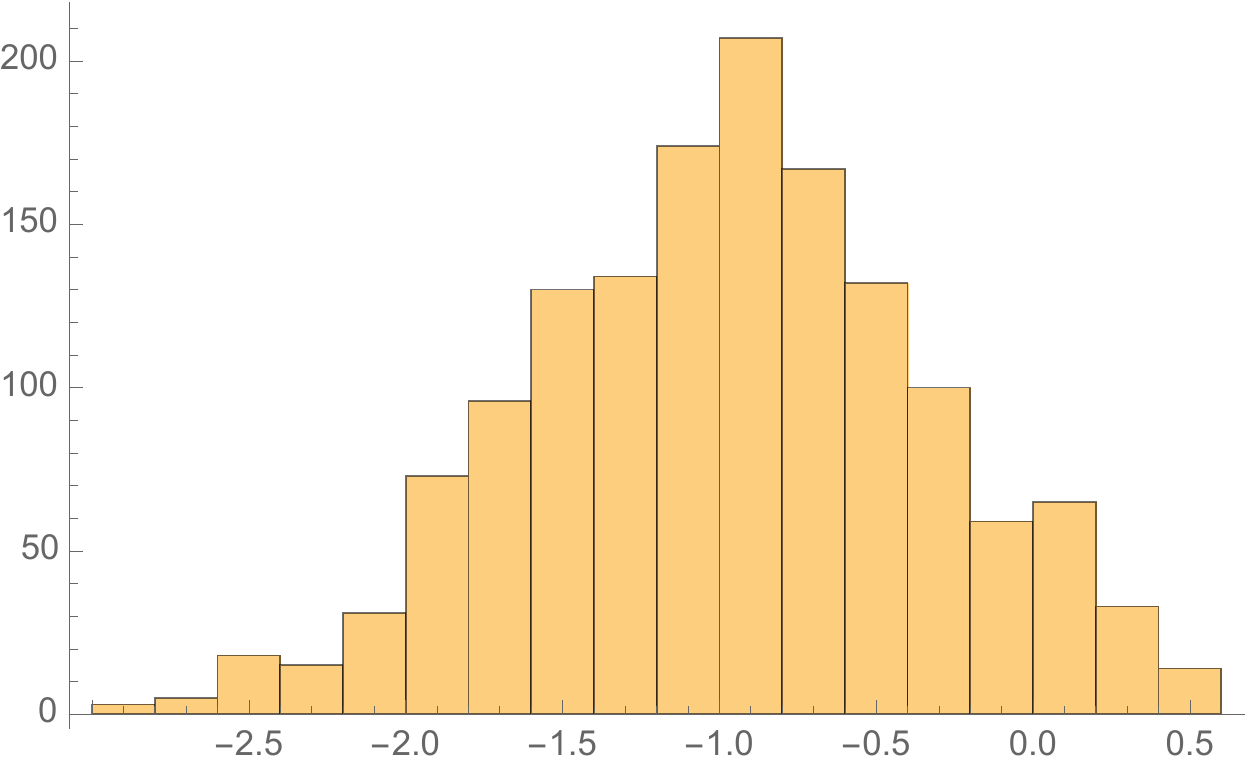}
\caption{\label{fig:histo3}Histogram of log $C(f)$}
\end{subfigure}
\hfill
\begin{subfigure}[b]{0.43\textwidth}
\includegraphics[width=\textwidth]{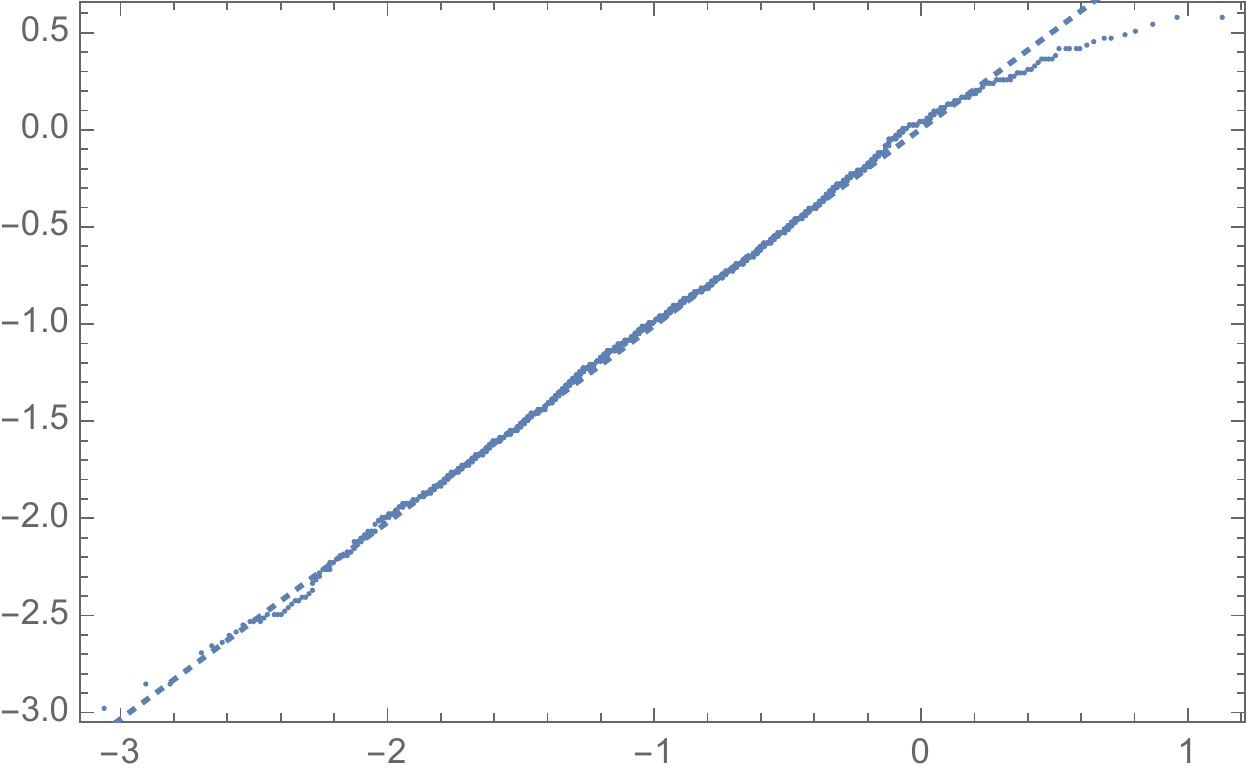}
\caption{\label{fig:qq3}quantile-quantile plot of log $C(f)$ }
\end{subfigure}
\caption{Degree 3}
\end{figure}

\begin{figure}[!tbp]
\begin{subfigure}[b]{0.43\textwidth}
\includegraphics[width=\textwidth]{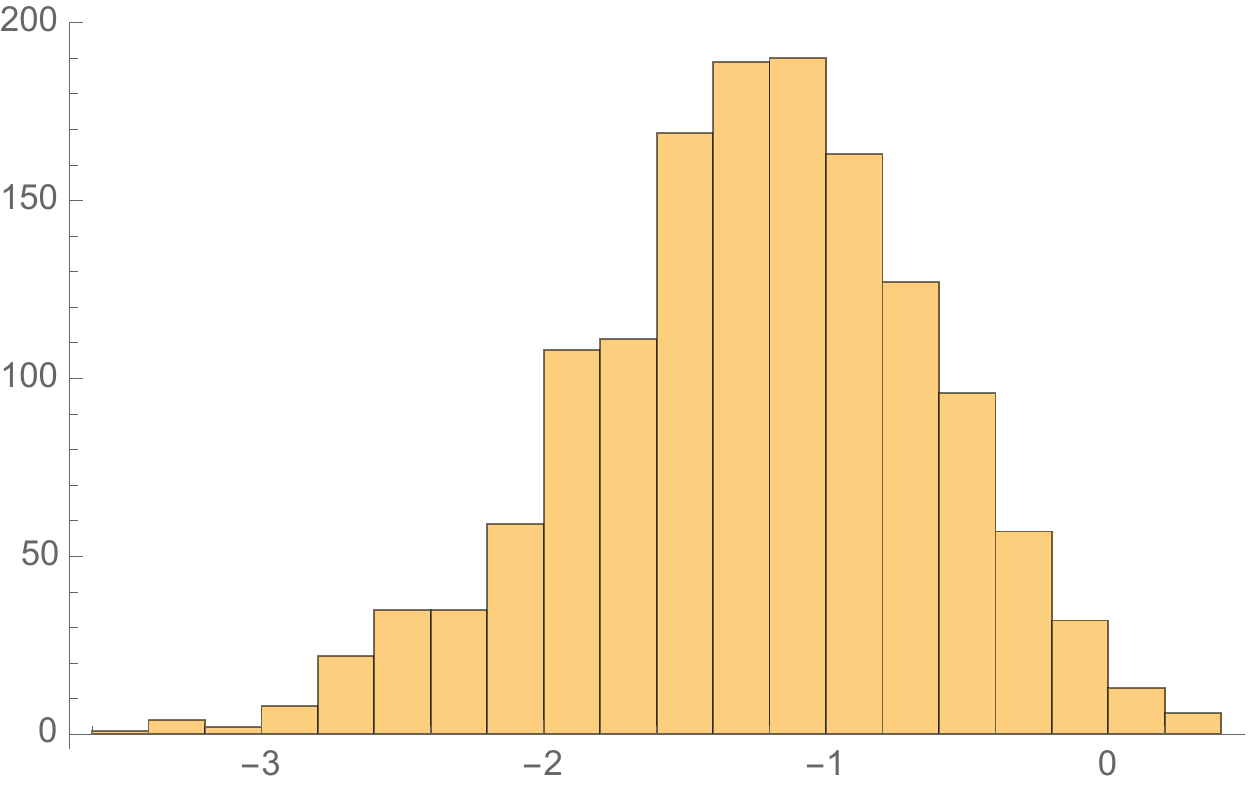}
\caption{\label{fig:histo4}Histogram of log $C(f)$}
\end{subfigure}
\hfill
\begin{subfigure}[b]{0.43\textwidth}
\includegraphics[width=\textwidth]{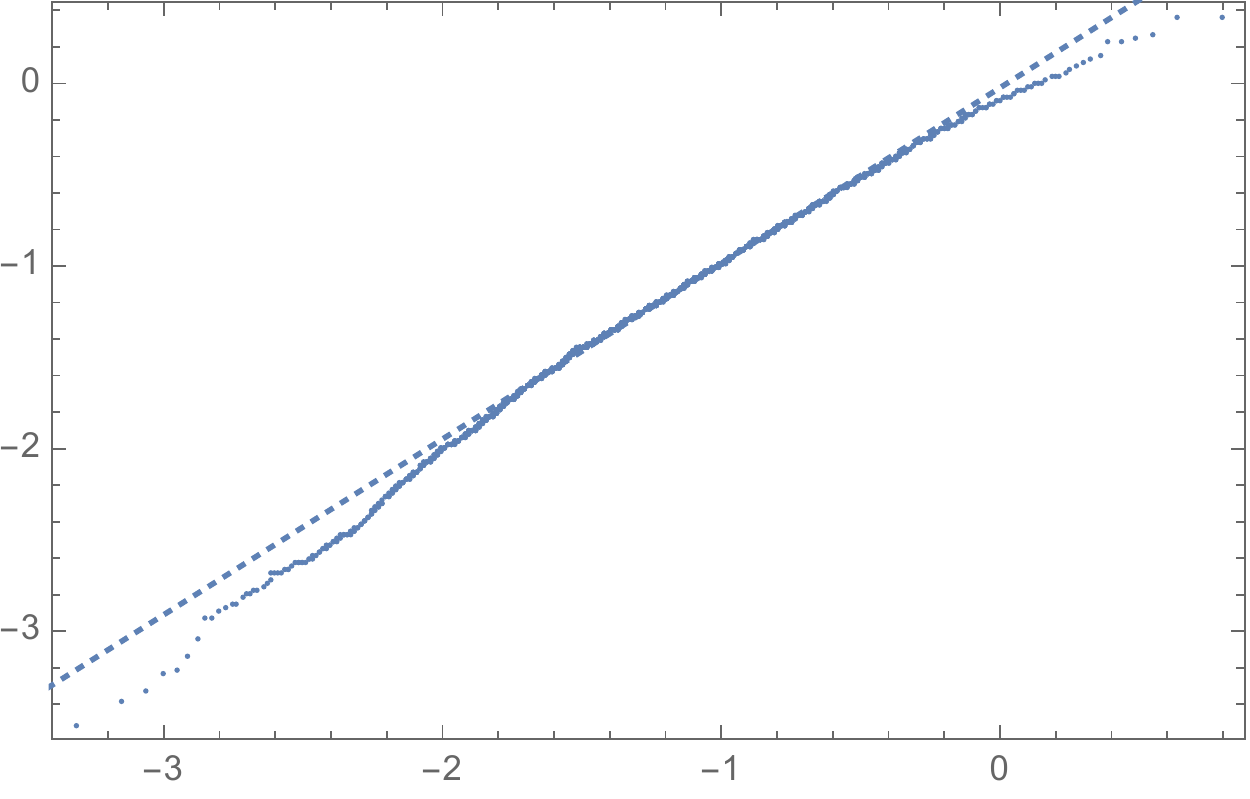}
\caption{\label{fig:qq4}quantile-quantile plot of log $C(f)$ }
\end{subfigure}
\caption{Degree 4}
\end{figure}

\begin{figure}[!tbp]
\begin{subfigure}[b]{0.43\textwidth}
\includegraphics[width=\textwidth]{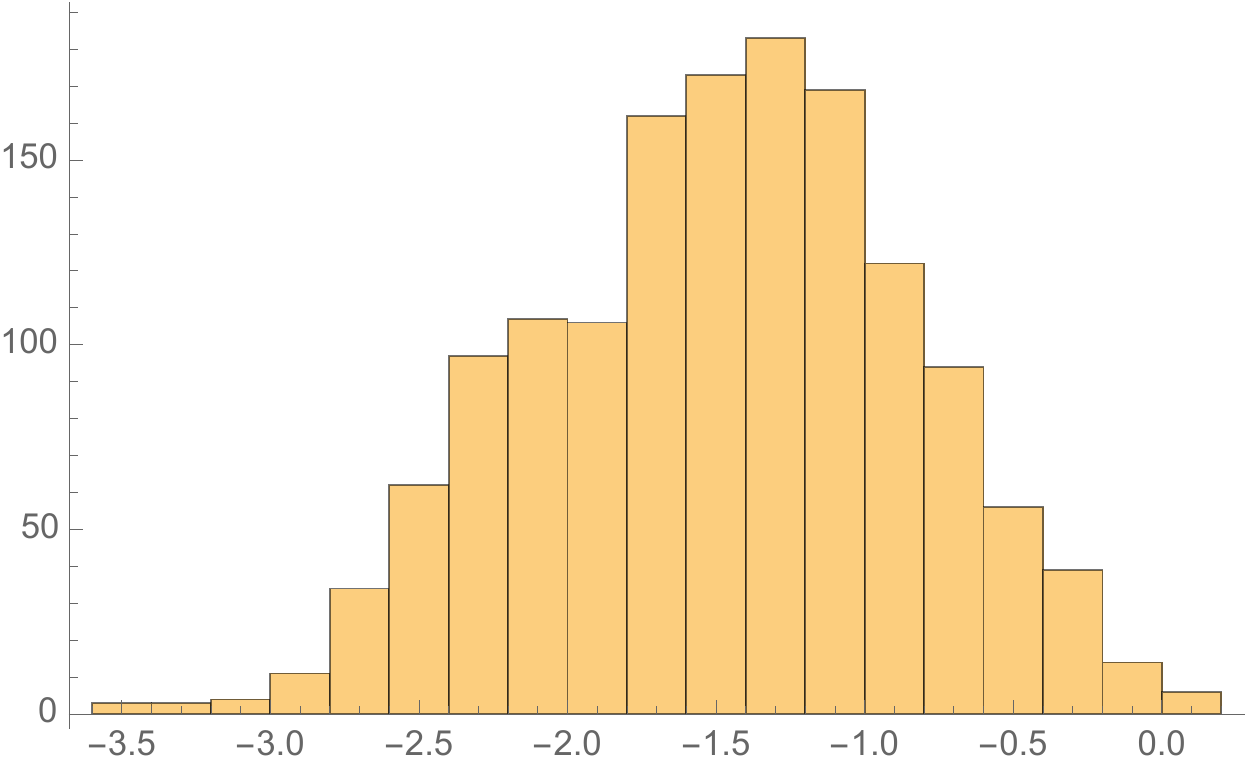}
\caption{\label{fig:histo5}Histogram of log $C(f)$}
\end{subfigure}
\hfill
\begin{subfigure}[b]{0.43\textwidth}
\includegraphics[width=\textwidth]{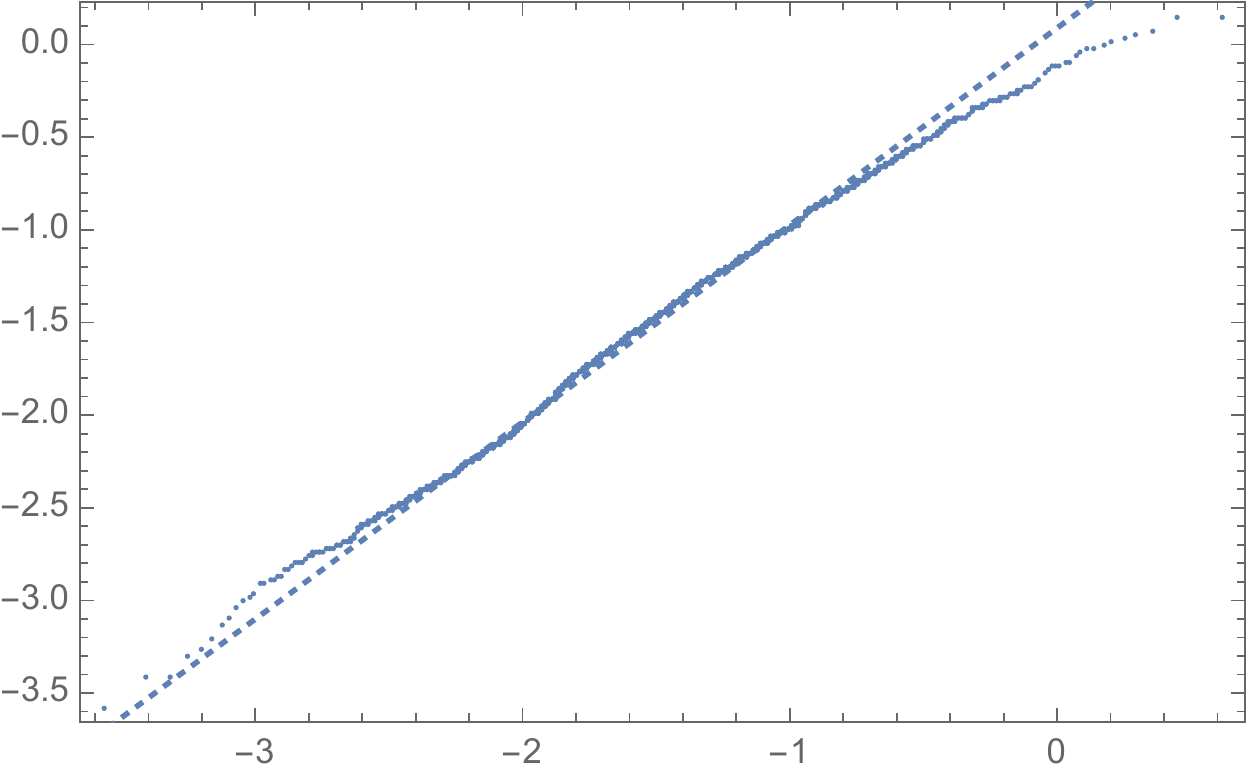}
\caption{\label{fig:qq5}quantile-quantile plot of log $C(f)$ }
\end{subfigure}
\caption{Degree 5}
\end{figure}

\begin{figure}[!tbp]
\begin{subfigure}[b]{0.43\textwidth}
\includegraphics[width=\textwidth]{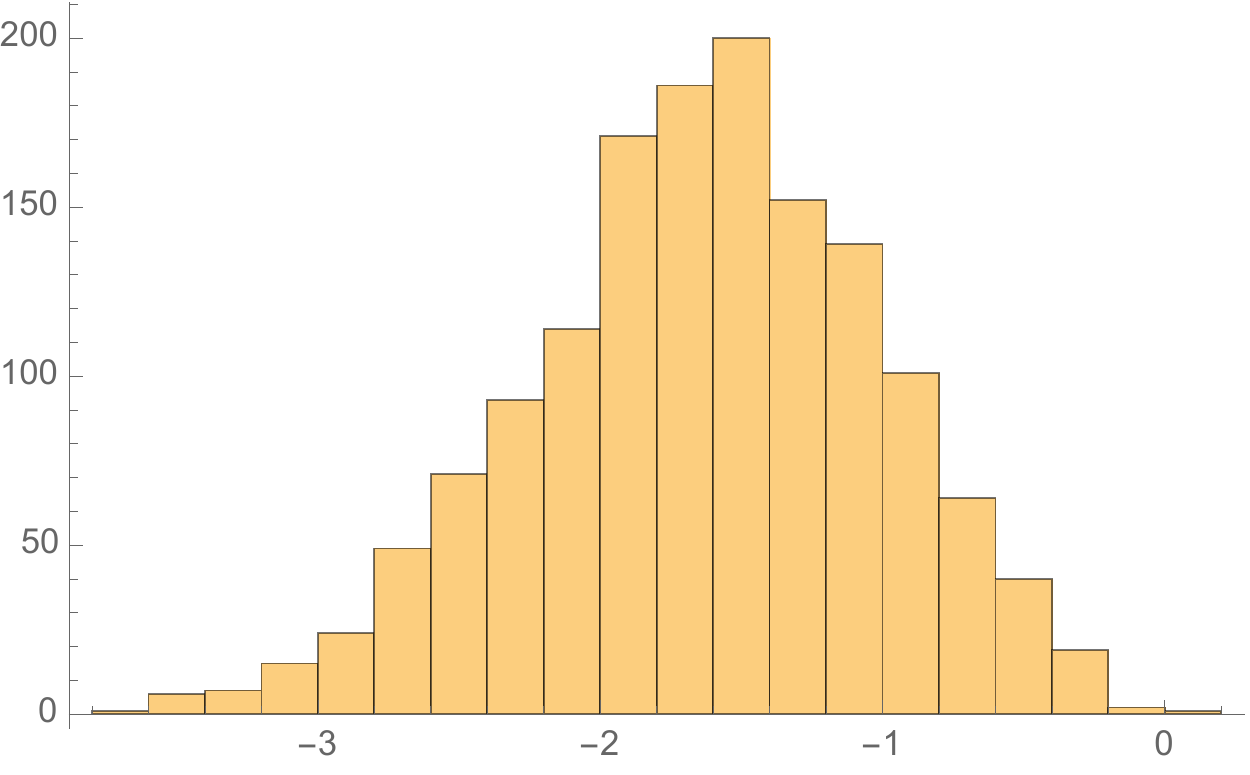}
\caption{\label{fig:histo6}Histogram of log $C(f)$}
\end{subfigure}
\hfill
\begin{subfigure}[b]{0.43\textwidth}
\includegraphics[width=\textwidth]{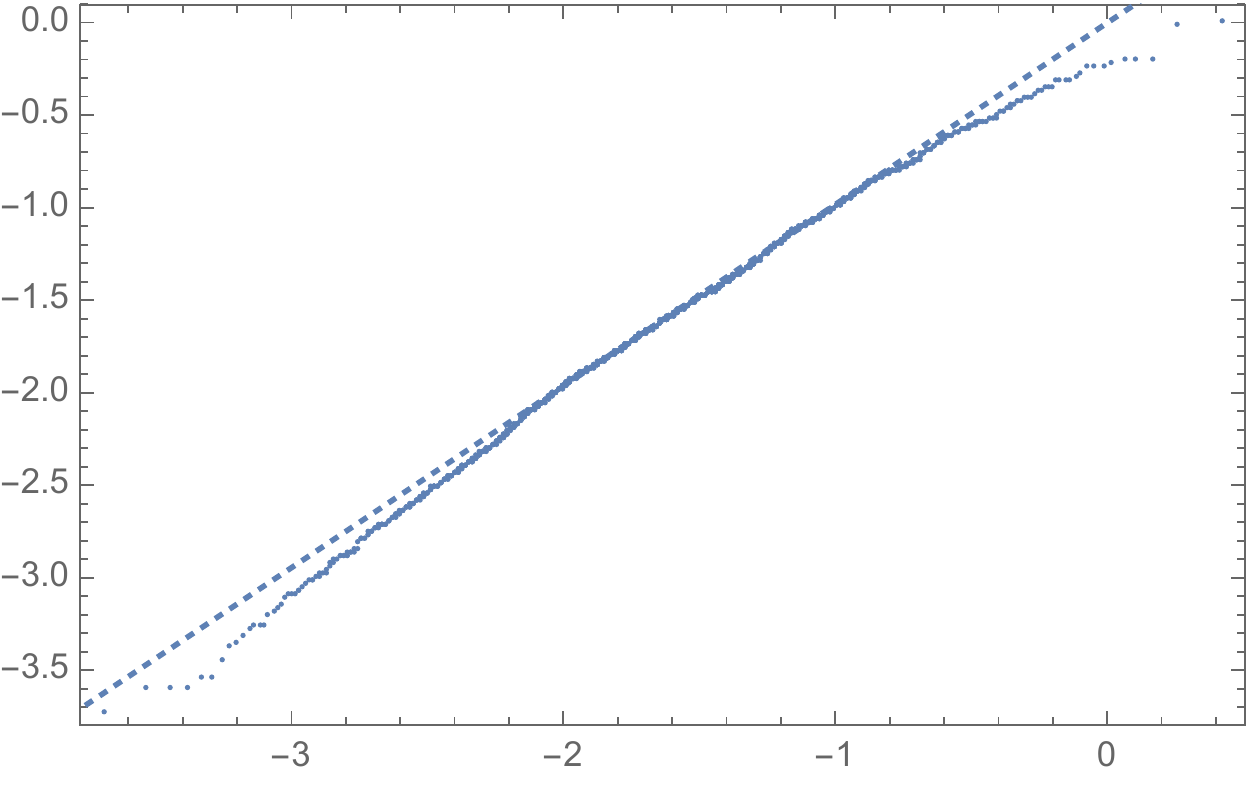}
\caption{\label{fig:qq6}quantile-quantile plot of log $C(f)$ }
\end{subfigure}
\caption{Degree 6}
\end{figure}


\begin{figure}[!tbp]
\begin{subfigure}[b]{0.43\textwidth}
\includegraphics[width=\textwidth]{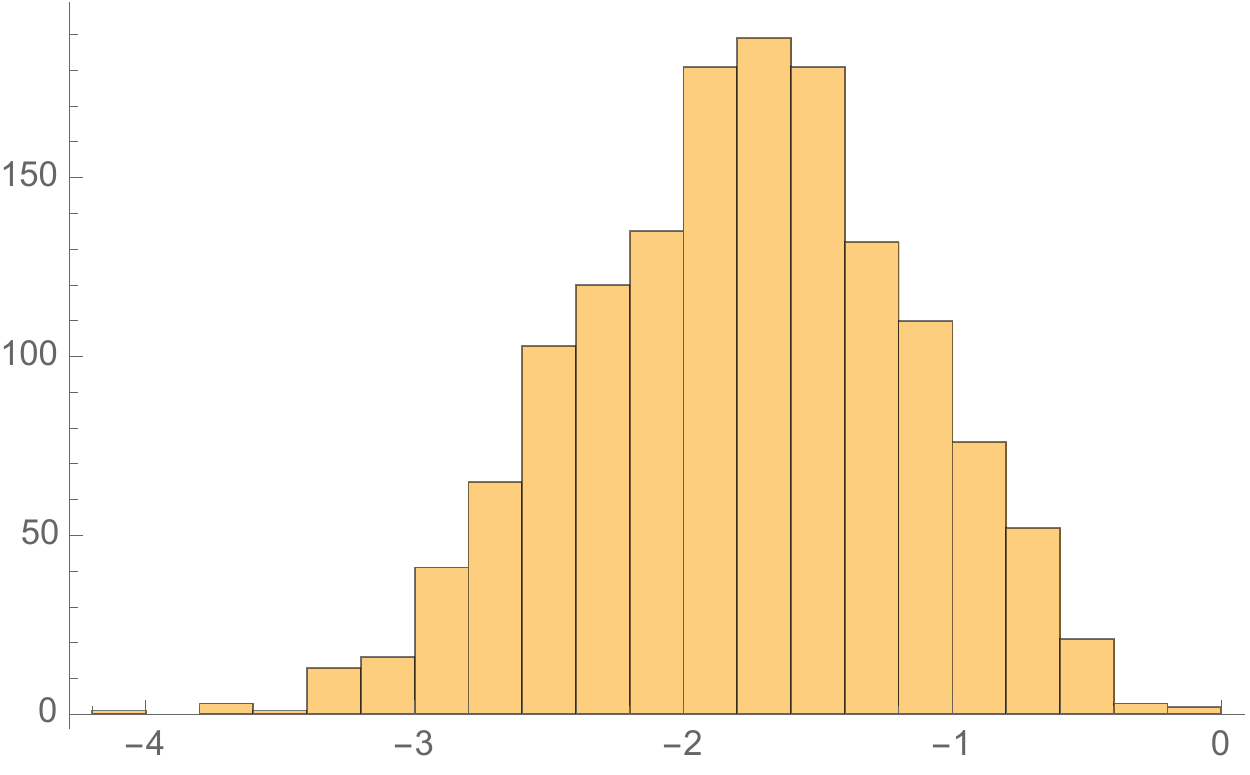}
\caption{\label{fig:histo7}Histogram of log $C(f)$}
\end{subfigure}
\hfill
\begin{subfigure}[b]{0.43\textwidth}
\includegraphics[width=\textwidth]{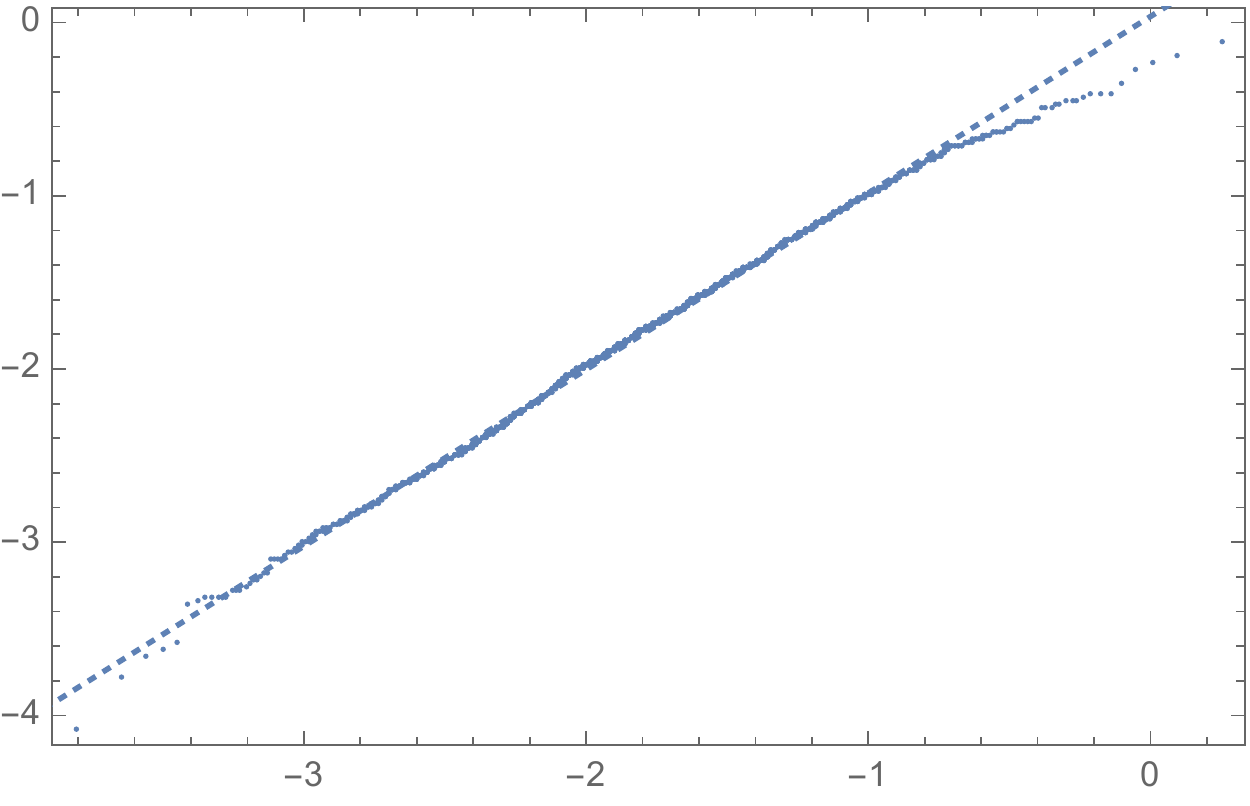}
\caption{\label{fig:qq7}quantile-quantile plot of log $C(f)$ }
\end{subfigure}
\caption{Degree 7}
\end{figure}

\begin{figure}[!tbp]
\begin{subfigure}[b]{0.43\textwidth}
\includegraphics[width=\textwidth]{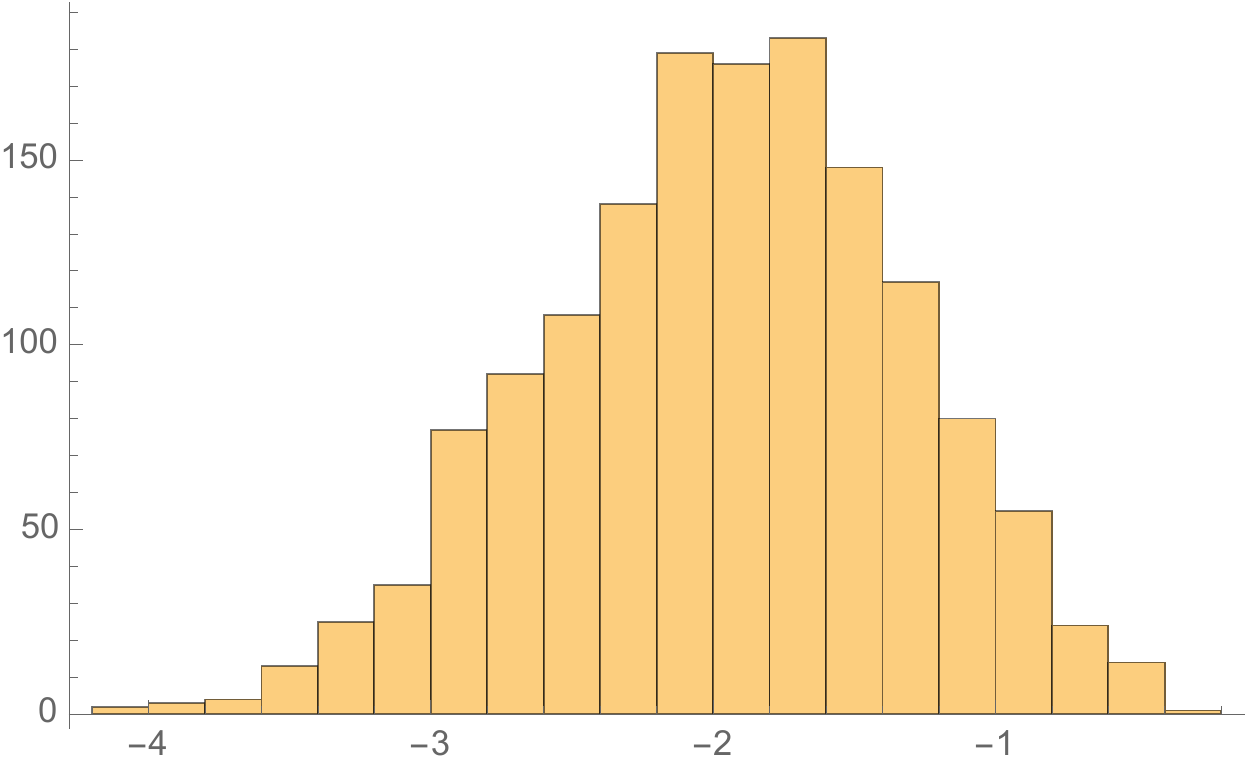}
\caption{\label{fig:histo8}Histogram of log $C(f)$}
\end{subfigure}
\hfill
\begin{subfigure}[b]{0.43\textwidth}
\includegraphics[width=\textwidth]{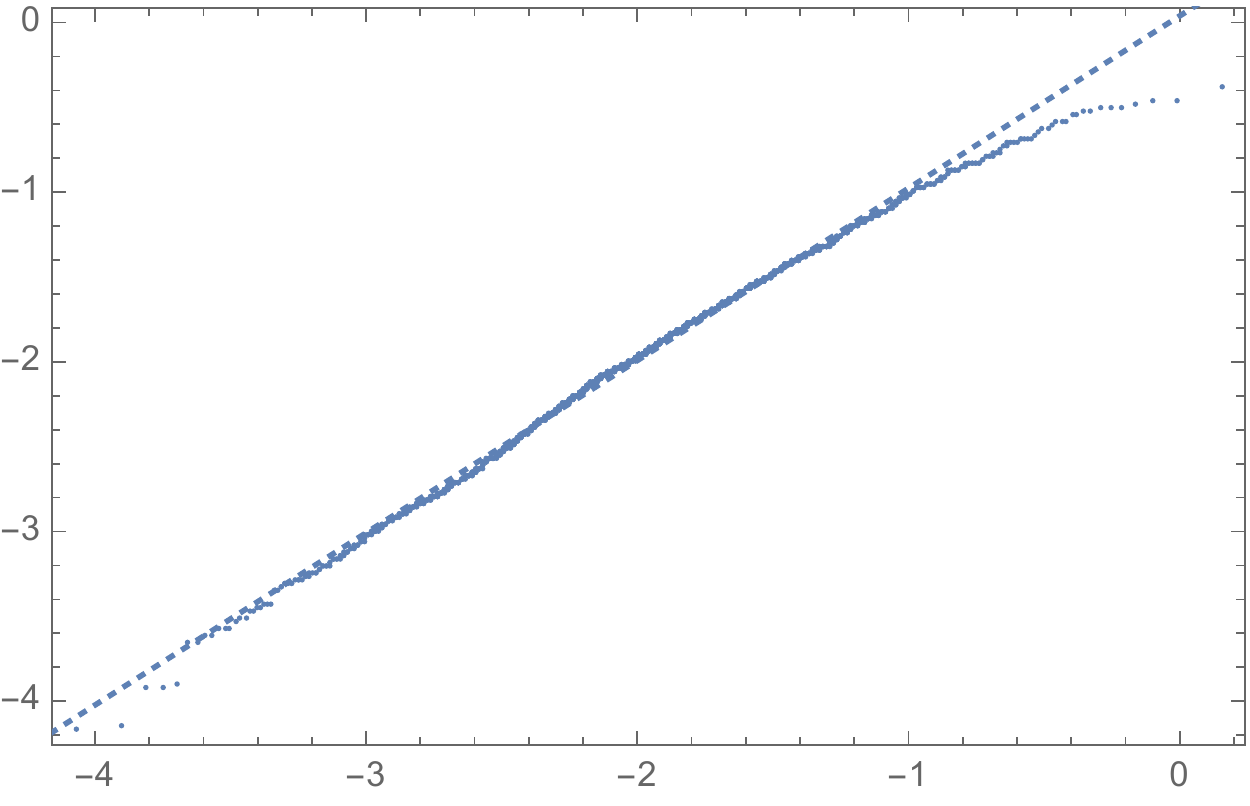}
\caption{\label{fig:qq8}quantile-quantile plot of log $C(f)$ }
\end{subfigure}
\caption{Degree 8}
\end{figure}

\section{How prime-rich can a polynomial be?}
As we have found (at least experimentally), the average value of $C(f)$ is $1,$ which means that a garden-variety monic (irreducible) polynomial of degree $d,$ when given an argument of size around $N,$ is about as likely to give a prime value as a number of size around $N^d.$ We say that a polynomial is \emph{prime-rich} if it is much more likely to give prime values (which means that $C(f) > 1$). Of course, since just over a quarter of all polynomials fail the Bunyakovsky condition (the probability that a polynomial with content $1$ in $\mathbb{Z}[x]$ gives even values always is around $1/4,$ odd values always is $1/27,$ and so on), the average Bunyakovsky polynomial is somewhat prime rich, but how well can we do? To find the answer, we computed $5000$ random polynomials of coefficients bounded above by $5000$ of each of the degrees $2, 3, 4, 5,$ and found the ones most prime-rich. Below, we give the three richest for each of the degrees we computed, but not before asking:
\begin{question}
\label{isbounded}
Is $C(f)$ bounded above (fixing degree)?
\end{question}
Of course, if the distribution of $C(f)$ is genuinely log-normal, then the answer is \textbf{NO.}

One place to look for an answer is in the proof that the product defining $C(f)$ actually converges. The proof (see 
\cite{bateman1962heuristic}) uses Landau's Prime Ideal Theorem to observe that 
\[
\sum_{p<x} \frac{n_p(f)}{p} = \log \log x + A_f + o(1),
\]
for \emph{any} irreducible $f$ (the constant $A_f$ for $f(x) = x$ is the so-called Mertens-Meissel constant $M$).
In turn, that implies that 
\[
\lim_{x\rightarrow \infty}\sum_{p<x} \frac{n_f(p) -1}p = A_f -M,
\]
which immediately implies that the product converges, and indicates that the size of $C(f)$ is controlled by the magic constant $A_f,$ which does not seem to help that much.

Note that if we had some uniform estimate on the convergence speed of the Euler product, then the answer to the Question \ref{isbounded} is clearly negative:
simply take a polynomial which has no roots for the first $k$ primes (this is easy to do by Chinese remaindering) - the product of the first $k$ terms of the defining product is going to be of order of $\log k$ (while the coefficients of the polynomial with given reductions mod $p$ will be of order of $e^k$) - a uniform estimate on the remainder would finish the job. However, we are not aware of any such uniformity result (notice that it would be enough to have uniformity \emph{on average}). If we remove the condition of being monic, then the obvious candidate (without Chinese remaindering) is the polynomial
\[
P_{n, k}(x) = 1+  x^k \prod_{i=1}^n p_i.
\]
A numerical experiment indicates that $C(P_{30, 2}) \approx 9.5.$
Modulo all of the above, one can ask a more detailed question:
\begin{question}
\label{morebound}
If we look at all $f \in \mathbb{Z}[x]$ of some fixed degree with coefficients bounded (in absolute value) by $N,$ is it true that the maximal $C(f)$ is of order $\log \log N?$
\end{question}
One can look further, and look at whether more sophisticated methods of computing $C(f)$ than just multiplying out the beginning of the Euler products give us any clue. Such methods are described in a number of very nice papers by Nobushige Kurokawa\footnote{The author would like to thank Keith Conrad for pointing these out to him.} -- \cite{kuro1,kuro2,kuro3,kuro4} - Kurokawa shows how to represent $C(f)$ as a product of Artin $L$-functions. Again, it is not clear how to leverage this to get a bound.

The other obvious question is:
\begin{question}
What do prime-rich polynomials have in common?
\end{question}
It is fairly clear (\emph{assuming the truth of the Bateman-Horn conjecture} ) that the constant term should be prime (or a product of large primes), and this is borne out by the results below.

Here are the winners (we give top three for each degree):

\begin{center}
{\tabulinesep=1.2mm
  \begin{tabu}{| c | r |}
    \hline
    $C(f)$ & $f(x)$ \\ \hline
   6.3722 & $x^2 - 2619 x + 1291$ \\ \hline
    6.36569 & $x^2 -2717 x - 1471$ \\ \hline
    6.23592 & $x^2 +2321 x + 911$ \\ \hline
   5.51225 & $x^3+3914 x^2-3485 x+2773$ \\ \hline
   5.37671 & $x^3-611 x^2-424 x-2999$ \\ \hline
   5.35378 & $x^3+707 x^2-800 x-509$ \\ \hline
   6.19895 & $x^4-1065 x^3+409 x^2-265 x+817$ \\ \hline
  5.74642 & $x^4+254 x^3-3125 x^2-1204 x-2609$ \\ \hline
   5.65711 & $x^4+4590 x^3-4932 x^2+2061 x-1289$ \\ \hline
  6.27693& $x^5-2127 x^4-1190 x^3-2317 x^2-1499 x-2257$ \\ \hline
     6.15153 & $x^5-4686 x^4+2812 x^3-4475 x^2-1714 x+3389$ \\ \hline
   5.74885 & $x^5-3738 x^4-150 x^3-4819 x^2-2748 x+307$ \\ \hline
    6.86439 & $x^6+2697 x^5+3377 x^4+2484 x^3-2700 x^2+1587 x+1831$ \\ \hline
   6.28748 & $x^6+4705 x^5+232 x^4-3661 x^3+3063 x^2-820 x-3533$ \\ \hline
    6.12438 & $x^6+1138 x^5-2757 x^4-4376 x^3+954 x^2+4060 x-2729$ \\ \hline
     5.72267 & $x^7+1965 x^6+2378 x^5-2384 x^4-1298 x^3-600 x^2-2610 x-1249$ \\ \hline
    5.62225 & $x^7-4618 x^6+4170 x^5+4299 x^4+1447 x^3+4695 x^2+2032 x-4387$ \\ \hline
    5.54845 & $x^7+704 x^6-2286 x^5-1938 x^4-462 x^3+2470 x^2+4241 x-2179$ \\ \hline
  \end{tabu}
  }
\end{center}
\bibliographystyle{plain}
\bibliography{bh}
\end{document}